\documentclass[journal]{IEEEtran}
\IEEEoverridecommandlockouts
\usepackage{lineno}
\usepackage[cmex10]{amsmath}
\usepackage{amscd,amsfonts,amsmath,amssymb,mathrsfs,pifont,stmaryrd,tipa}
\usepackage{array,cases,dsfont,graphicx,texdraw}
\usepackage{graphics}
\usepackage{epstopdf}
\usepackage{epsfig}
\usepackage{cite}
\usepackage{mathbbold}
\usepackage{stfloats}
\usepackage{subcaption}
\usepackage{subfig}
\usepackage{mathtools}
\usepackage{wrapfig}
\usepackage{enumitem}
\usepackage{amsfonts}
\usepackage{amssymb}
\usepackage{graphicx}%
\usepackage[ruled, linesnumbered]{algorithm2e}
\usepackage{xcolor}

\usepackage{amsthm}       

\usepackage{caption}
\usepackage[labelfont={color=black}]{caption}

\newtheorem{theorem}{\bf{Theorem}}

\newtheorem{assumption}{Assumption}

\newtheorem{lemma}{\bf{Lemma}}

\newtheorem{remark}{\bf{Remark}}

\allowdisplaybreaks
\begin{document}
	
	\title{Clipping-Free Nash Equilibrium Seeking in Heavy-Tailed   Games via Median-of-Means}
	\author{Chao Sun, Huiming Zhang, Bo Chen, Jianzheng Wang, Zheming Wang, Li Yu
		\thanks{This work was supported in part by the Joint Funds of the National Natural Science Foundation of China under Grant U24A20258, in part by the Zhejiang Provincial Natural Science Foundation of China under Grant LRG25F030001, in part by the funding of Leading Innovative and Entrepreneur Team Introduction Program of Zhejiang under Grant 2023R01006, and in part by the Fundamental Research Funds for the Provincial Universities of Zhejiang under Grant RF-C2023007. (Corresponding author: Bo Chen.  Email: bchen@aliyun.com.) 
			
			Chao Sun, Bo Chen, Jianzheng Wang, Zheming Wang, and Li Yu are with the  Department of Automation, Zhejiang University of Technology, China, and Zhejiang Key Laboratory of Intelligent Perception and Control for Complex Systems. Huiming Zhang is with the Institute of Artificial Intelligence, Beihang University, China.}}
	\maketitle
	
	\begin{abstract}
	This paper studies Nash equilibrium seeking for stochastic games under heavy-tailed gradient noise. 
	The noise is assumed to have a finite $\delta$-th moment with $1<\delta\leq 2$, which allows infinite variance. 
	To obtain robust gradient estimates, we adopt the median-of-means (MoM) method in robust estimation. 
	At each iteration, samples are split into blocks, the gradients in each block are averaged, and the median of these block means is used to update the actions. 
	Compared with gradient clipping, MoM does not require a preset clipping threshold. 
	It is also robust to outlying and corrupted gradient samples. 
	Under standard assumptions, we prove convergence of the proposed algorithm and derive its convergence rate. 
	To reduce the bias caused by asymmetric noise, we further design an online bias-correction scheme. The simulation results show that the proposed methods performs much better than the clipping-based methods, especially for noise with a symmetric distribution.
	\end{abstract}
	

	\thispagestyle{empty}
\section{Introduction}

Game theory provides a mathematical framework for studying interactions among rational decision-makers. In non-cooperative games, a Nash equilibrium \cite{Nash} is a stable state where no player can improve their payoff by changing only their own strategy, given the strategies of others. This concept has become a core tool for analyzing competitive problems in many fields, such as   mathematics \cite{mazalov2014mathematical}, economics \cite{gibbons1992game}, and energy systems \cite{mansouri2023three}. Nash equilibrium seeking in non-cooperative games has been extensively studied under various settings, including distributed optimization and continuous-time dynamics \cite{Ye, Xie, liang2017distributed,  frihauf2011nash,   pang2020distributed, de2019continuous, tatarenko2018learning, yi2019operator, he2025distributed, lu2018distributed,  feng2025adaptively1}. However, real decision-making environments often involve uncertainty. The payoff functions of players may be affected by random factors, such as fluctuations in market demand or changes in environmental dynamics. Stochastic games  \cite{ 4610024,xu2010sample,koshal2012regularized,iusem2017extragradient,huang2022new,franci2021stochastic,beznosikov2022decentralized,lei2022distributed,yu2017distributed,franci2020distributed, zheng2022distributed} can handle this uncertainty. Each player aims to minimize an expected cost or maximize an expected payoff, which includes a random variable.

In fact, the distribution of gradient estimation errors caused by randomness is crucial for the performance of stochastic game algorithms. However, to ensure convergence, existing algorithms usually consider only Gaussian-type data, mainly because it is convenient for mathematical analysis. Many recent studies have shown that the Gaussian distribution is too ideal to describe the data in real tasks, and heavy-tailed distributions, which are broader than the Gaussian distribution, are much closer to reality \cite{simsekli2019tail, zhang2020adaptive,gorbunov2020stochastic,kornilov2023gradient}. For example, the authors in \cite{simsekli2019tail} conducted extensive experiments on multiple architectures and datasets, and found that in all configurations, the estimated tail index is far below 2, indicating clear heavy-tailed behavior. The authors in \cite{zhang2020adaptive} from MIT and Google pointed out that when pre-training the large language model BERT on the Wikipedia dataset, the resulting stochastic gradients were heavy-tailed data with unbounded variance. 

While in the most related game studies, the authors in \cite{gorbunov2022clipped} found that when training practical game problems such as generative adversarial networks (GANs), the gradient noise indeed exhibits heavy-tailed characteristics, and introducing gradient clipping significantly improves algorithm performance, e.g., WGAN-GP's FID drops from 67.37 to 19.65, and StyleGAN2 goes from being completely untrainable to generating meaningful images.

Currently, there are only limited studies on the Nash equilibrium seeking problem for heavy-tailed games. The authors in \cite{gorbunov2022clipped, sadiev2023high} proposed gradient-clipping based approaches for heavy-tailed variational inequality problems, which can be applied to game issues. In \cite{sun2025distributed1}, we proposed a distributed gradient-clipping based method for non-cooperative games under the graph framework. However, these methods all rely on gradient clipping to tame the heavy-tailed noise. While gradient clipping is effective, it requires careful tuning of the clipping threshold and may introduce bias. 

To overcome these limitations, we apply the  MoM  technique in robust estimation, which provides robust gradient estimates without the need for clipping thresholds. This method has several practical benefits. First, it removes the requirement to carefully tune a clipping threshold, a hyperparameter that often needs domain knowledge and is sensitive to the unknown noise level. Second, unlike clipping which introduces a systematic bias, the median-of-means estimator can give an unbiased or nearly unbiased estimate under symmetric heavy-tailed noise, helping to preserve the true gradient direction. Third, it is naturally resilient to  corrupted or adversarial gradient updates, because it aggregates information by blocks and uses the median as a central tendency and ignores a bounded proportion of outliers. These properties make the algorithm more reliable in stochastic game settings.

The main contributions of this work are summarized as follows:

(1) A MoM-based Nash equilibrium seeking algorithm is proposed. It works under heavy-tailed gradient noise with only a finite $\delta$-th moment ($1<\delta\le 2$). Unlike gradient clipping, this method requires no preset clipping threshold, while naturally defending against malicious gradient attacks.
	
(2) We provide the convergence analysis for MoM applying to Nash equilibrium seeking problems. Furthermore, an almost sure convergence rate is derived and the influence of the heavy-tail exponent is shown.
	 
	(3) To address the systematic bias of the plain MoM estimator under asymmetric noise, an online bias correction strategy is introduced and a strict convergence proof is provided.  

\textbf{Notations}: Throughout this paper, 0 is the real number 0 or a zero vector with appropriate dimension. $\mathbb{R}$ and $\mathbb{R}^N$ represent the real number set and the $N$-dimensional real vector set, respectively. $\left\Vert e \right\Vert$ is the 2-norm of vector $e$. $\left\vert \cdot \right\vert$ is the absolute value. $\lambda_{\min}\{\cdot\}$ is the minimal eigenvalue of a matrix. $\mathbb{P}_\Omega[\cdot]$ is the Euclidean projection of a vector onto a set $\Omega$. $\mathbb{E}[\cdot]$ is the expectation of a random variable.  $\nabla_{x}f(y)$ is the gradient of a function $f(\cdot)$ with respect to $x$ at point $y$. $\operatorname{median}\{x_1,\cdots, x_N\}$ is the median of the $N$ variables where if $N$ is even, the average of the two middle values is taken. $\lfloor\cdot \rfloor$ represents the floor function and $\lceil \cdot \rceil$ represents the ceil function.


\section{Problem Formulation\label{S3}}

Consider a non-cooperative game comprised of $N>1$ players. The $i$-th player, $i=1,\cdots,N$, aims to solve the following stochastic optimization problem
\begin{align}
	\min_{x_i\in \Omega_i} J_i(x_i,x_{-i}):=  \mathbb{E}_{\xi_i}[f_i(x_i,x_{-i},\xi_i)], \label{system1}
\end{align}
where $x_i\in\Omega_i\subseteq \mathbb{R}$ is the action of player $i$, $x_{-i}\in\Omega_{-i}\subseteq \mathbb{R}^{N-1}$ is the action of players except $i$, $\xi_i\textcolor{black}{\in\mathbb{R}}$ is a local random variable, $\Omega_i$ is a local constraint set. 

In the following, we write $J_i(x_i,x_{-i})$ as $J_i(x)$ for brevity. The following assumptions are made.

\begin{assumption}
	$\Omega:=\Omega_1\times \Omega_2\times\cdots\times\Omega_N$ is a nonempty, convex and compact set. \label{boundedness}
\end{assumption}

\begin{assumption}
 $J_i(x)$ is continuously differentiable \textcolor{black}{and convex in $x_i$ for every fixed $x_{-i}\in\Omega_{-i}$}. \label{conti}
\end{assumption}

\begin{assumption}
	The pseudo-gradient mapping $F(x):=[\nabla_{x_1}J_1(x),$ $\cdots, \nabla_{x_N}J_N(x)]^\top\textcolor{black}{\in\mathbb{R}^N}$ is strongly monotone   with modulus $\mu$, i.e., there exists a positive constant $\mu$ such that $(F(x)-F(y))^\top(x-y)\geq \mu||x-y||^2$ for all $x,y\in \Omega$.\label{strongly}
\end{assumption}

Under Assumptions \ref{boundedness}, \ref{conti} and \ref{strongly}, there exists a unique Nash equilibrium $x^*$ \cite{scutari2014real}. Furthermore, $F(x^*)(x-x^*)\geq 0$.

\begin{assumption}
	$F(x)$ is $L$-Lipschitz continuous, i,e. there is a constant $L>0$ such that $\Vert F(x)-F(y) \Vert \leq L \Vert x-y \Vert $  for all $x,y\in \Omega$. \label{Lipschitz}
\end{assumption}

\color{black}
\begin{remark}
	Assumptions \ref{boundedness}--\ref{Lipschitz} are are standard assumptions in the Nash equilibrium seeking literature. For example, the strong monotonicity appears in \cite{Ye,pang2020distributed}.
\end{remark}

\color{black}

\section{Median-of-Means Nash Equilibrium Seeking} \label{sec_mom}

\subsection{Algorithm Design} \label{sec_al1}

At each iteration $k = 0,1,2,\dots$, every player $i$ maintains a local action
$x_{i,k}\in\Omega_i$. Let $\mathbf{x}_{ k}=[x_{1, k},\cdots, x_{N, k}]^\top\textcolor{black}{\in\mathbb{R}^N}$ represent the action vector at step $k=0,1,\cdots$. The player updates its action according to the following four steps.

\begin{enumerate}[leftmargin=*, label=\textbf{Step \arabic*:}]
	\item \textbf{Sample collection.}  Draw $m_k$ i.i.d.\ samples
	$\xi_{i,k}^1,\dots,\xi_{i,k}^{m_k}$ and compute the stochastic gradients
	\begin{equation}
		g_{i,k}^j = \nabla_{x_i} f_i(\mathbf{x}_k, \xi_{i,k}^j),\quad j=1,\dots,m_k. \label{step1}
	\end{equation}
	
	\item \textbf{Partition and within-block averaging.} Partition the $m_k$ samples into  $b_k$ disjoint blocks with  size $s_k$ according to the subsequent theorems. For each block
	$B_\ell, \ell=1,\cdots, b_k$, compute the empirical mean
	\begin{equation}
		\bar{g}_{i,k}^\ell = \frac{1}{s_k}\sum_{j\in B_\ell} g_{i,k}^j.\label{step2}
	\end{equation}
	
	\item \textbf{Median aggregation.} Compute the sample median of the $b_k$ block means:
	\begin{equation}
		\hat{g}_{i,k} = \operatorname{median}\{\bar{g}_{i,k}^1,\dots,\bar{g}_{i,k}^{b_k}\}.\label{step3}
	\end{equation}
	If $b_k$ is even, the average of the two middle values is taken.
	
	\item \textbf{Action update.} Perform a projected gradient step:
	\begin{equation}
		x_{i,k+1} = \mathbb{P}_{\Omega_i}\bigl[x_{i,k} - \alpha_k \hat{g}_{i,k}\bigr],\label{step4}
	\end{equation}
	where $\alpha_k > 0$ is the step-size.
\end{enumerate}

\subsection{Assumptions on the Per-Sample Gradient Noise}

The algorithm in Section~IV.A uses a stochastic gradient oracle. At each step $k$,
the oracle gives a noisy gradient value $\nabla_{x_i} f_i(\mathbf{x}_k, \xi_i)$
for each player $i$. 

Let $\mathcal{F}_k$ be the $\sigma$-algebra that contains all the randomness up to
step $k-1$. Thus, the current action $\mathbf{x}_k$ is known given $\mathcal{F}_k$.
For player $i$ at step $k$, let the $j$-th sample be $\xi_{i,k}^j$. The per-sample
noise is defined as
\begin{equation}
	\epsilon_{i,k}^j = g_{i,k}^j - \nabla_{x_i} J_i(\mathbf{x}_k)
	= \nabla_{x_i} f_i(\mathbf{x}_k, \xi_{i,k}^j) - \nabla_{x_i} J_i(\mathbf{x}_k).
\end{equation}
Here, $j=1,\dots,m_k$ is the index of the sample inside step $k$.

The following two assumptions are made for the noise.

\begin{assumption}
	\label{ass:unbiased}
	For every player $i$, every step $k\ge0$, and every sample index $j$, we have
	\begin{equation}
		\mathbb{E}\bigl[ \epsilon_{i,k}^j \mid \mathcal{F}_k \bigr] = 0 \quad \text{a.s.}
	\end{equation}
\end{assumption}

\begin{assumption}
	\label{ass:heavytail}
	There exist numbers $1<\delta\le 2$ and $\nu>0$ such that for all $i$, $k$, and $j$,
	\begin{equation}
		\mathbb{E}\bigl[ |\epsilon_{i,k}^j|^\delta \mid \mathcal{F}_k \bigr] \le \nu^\delta \quad \text{a.s.}
	\end{equation}
	which means that the $\delta$-th moment of the noise is bounded, even if the variance
	(when $\delta<2$) may be infinite.
\end{assumption}

\subsection{Key Lemmas} \label{sub_key}

The following lemma is a key conclusion to prove the almost sure convergence of the algorithm, which is well-known in optimization theory.

\begin{lemma} [Lemma 3 of \cite{polyak} at Page 45]  Let $Y_k$ be a sequence  with $Y_k\geq 0$,
	\begin{align}
	Y_{k+1}  \leq\left(1-u_k\right) Y_k+\beta_k,
	\end{align}
	where $0 <u_k \leq 1, \beta_k \geq 0$,  $\sum_{k=0}^{+\infty} u_k=+\infty$, and $\lim _{k \rightarrow +\infty} \beta_k / u_k=0$. Then, $Y_k \rightarrow 0$. \label{a}
\end{lemma}

 A key result for MoM, from \cite{bubeck2013bandits}, provides a high-probability bound. We restate it here as Lemma 2 where we adapt the notation to our game-theoretic context and derive an additional conclusion.

\begin{lemma}  \label{lemmaz}
	Let $\alpha\in(0,1]$ and $v>0$. Let $Z_1,\cdots,Z_m$ be i.i.d. real-valued random variables with mean $\bar{\theta}$ and centered $(1+\alpha)$-th absolute moment bounded by $u$, i.e., 
	\begin{align}
		 \mathbb{E}[\vert Z_1-\bar{\theta}\vert^{1+\alpha}] \leq u.
	\end{align}
	Dividing $m$ samples into $b$ disjoint blocks and each block has $s$ samples, and the remaining samples are at most $s-1$.
	For any $\gamma\in(0,1)$ such that $m\geq 16\ln(e^{1/8}\gamma^{-1})+2$, let
		\begin{align}
		 b = \lfloor \min\{8\ln(e^{1/8} \gamma^{-1}), \frac{m}{2} \} \rfloor, s=\lfloor \frac{m}{b}\rfloor. \label{c1}
	\end{align} 
 For each block $\ell = 1,\cdots, b$, compute the empirical mean  $\bar{Z}_\ell = \frac{1}{s}\sum_{j\in B_\ell}  Z_j$. Let $\hat{\theta}=\text{median}(\bar{Z}_1,\cdots,\bar{Z}_b )$. Then, with probability at least $1-\gamma$,
	\begin{align}
		\hat{\theta} \leq \bar{\theta}+(12u)^{\frac{1}{1+\alpha}}\left(\frac{16\ln(e^{1/8} \gamma^{-1})}{m}\right)^{\frac{\alpha}{1+\alpha}}. \label{part1}
	\end{align}
	The same bound holds for the lower tail 
		\begin{align}
		\bar{\theta} \leq \hat{\theta}+(12u)^{\frac{1}{1+\alpha}}\left(\frac{16\ln(e^{1/8} \gamma^{-1})}{m}\right)^{\frac{\alpha}{1+\alpha}}. \label{part2}
	\end{align}
	\end{lemma}
	\begin{proof}
		Inequality  \eqref{part1} is the same as Lemma 2 of \cite{bubeck2013bandits}. 	Inequality  \eqref{part2} can be obtained by symmetry. The proof for \eqref{part2} is put at Appendix \ref{proof_lemmaz} for completeness.
	\end{proof}
	
By Lemma \ref{lemmaz}, we can get the following conclusion for the MoM gradient estimate error, which will be used in the convergence analysis.

\begin{lemma}[Conditional Tail Bound for Median-of-Means Gradient Estimate]
	\label{lem:cond_tail_bound}
	Let Assumptions \ref{ass:unbiased} and \ref{ass:heavytail} hold. For any iteration $k\ge 0$, any player $i\in\{1,\dots,N\}$,
	and any confidence level $\gamma\in(0,1)$, choose the sample size $m_k\geq 16\ln(e^{1/8}\gamma^{-1})+2$
	and the block parameters $b_k, s_k$ as deterministic integers satisfying condition \eqref{c1}
	with $m = m_k$. Define the error threshold
	\begin{align}
		\mathcal{E}_k(\gamma) := C_1 \left( \frac{16\ln(e^{1/8}\gamma^{-1})}{m_k} \right)^{\frac{\delta-1}{\delta}}, \label{13}
	\end{align}
	where $C_1 = (12\nu^\delta)^{1/\delta}$ and $\delta\in(1,2]$ is the moment exponent
	from Assumption \ref{ass:heavytail}. Then, the median-of-means gradient estimate $\hat{g}_{i,k}$
	satisfies the conditional probability bound
	\begin{align}
		\mathbb{P}\bigl( |\hat{g}_{i,k} - \nabla_{x_i}J_i(\mathbf{x}_k)| > \mathcal{E}_k(\gamma) \mid \mathcal{F}_k \bigr) \le 2 \gamma \quad \text{a.s.}
	\end{align}
\end{lemma}
\begin{proof}
See Appendix \ref{conditional}.
\end{proof}

\subsection{Convergence Theorem}

The following theorem establishes the almost sure convergence of the proposed algorithm.

\begin{theorem}[Almost Sure Convergence]
	\label{thm:convergence}
	Let Assumptions \ref{boundedness}--\ref{ass:heavytail} hold. Consider the algorithm described in Section \ref{sec_al1} with
	step-size sequence $\alpha_k = b (k+1)^{-a}$ for some $a\in(0,1]$ and $b>0$, and sample
	size $m_k = c\lceil (k+1)^\beta \rceil$ with any $\beta>0$ and integer $c\geq 1$. Choose the block parameters
	$b_k, s_k$ according to \eqref{c1} with $m=m_k$ and confidence level $\gamma_k = 1/(k+1)^2$.
	Then the sequence of action profiles $\{\mathbf{x}_k\}$ generated by the algorithm
	converges almost surely to the unique Nash equilibrium $x^*$, i.e.,
	\begin{align}
	\lim_{k\to\infty} \lVert \mathbf{x}_k - x^* \rVert = 0 \quad \text{a.s.}
	\end{align}
\end{theorem}

\begin{proof}
See Appendix \ref{proof_t1}.
\end{proof}

\subsection{Convergence Rate}
	
	The following  lemma  extends the classical Chung's lemma (Lemma 5 of \cite{polyak} at Page 46) to handle logarithmic factors, which will be used to establish the almost sure convergence rate.
	
	\begin{lemma}[Chung's Lemma with Logarithmic Factors]\label{lem:chung-log}
		Let $\{Y_k\}_{k=k_0}^\infty$ be a sequence of nonnegative real numbers. Suppose that there exist constants $r>p\ge 0$, $d>0$, $\tau\ge 0$ and an integer $k_0\ge 1$ such that for all $k\ge k_0$,
		\begin{align}
			Y_{k+1}\le \Big(1-\frac{r}{k}\Big)Y_k+\frac{d(\ln k)^\tau}{k^{p+1}}. \label{re_Y}
		\end{align}
		Then there exists an integer $K_g\geq k_0$ which depends on $p$, $\tau$, $r$, $k_0$ only and a constant $A:= \max\Big\{\frac{2d}{r-p},\ \frac{Y_{K_g}K_g^p}{(\ln K_g)^\tau}\Big\}+1>0$ such that for all $k\ge K_g$,
		\begin{align}
			Y_k\le A\frac{(\ln k)^\tau}{k^p}.
		\end{align}
	\end{lemma}
	
	\begin{proof}
	See Appendix \ref{appd_1}.
	\end{proof}

	We now apply Lemma~\ref{lem:chung-log} to derive the almost sure convergence rate stated in Theorem \ref{thm:convergence} for the case $a=1$. The rate can be stated as follows.

	\begin{theorem}[Almost Sure Convergence Rate] \label{thm:convergence_rate}
	Let Assumptions \ref{boundedness}--\ref{ass:heavytail} hold. Consider the algorithm described in Section \ref{sec_al1} with
	step-size sequence $\alpha_k = b (k+1)^{-1}$ for some  $b>0$, and sample
	size $m_k = c\lceil (k+1)^\beta \rceil$ with any $\beta>0$ and integer $c\geq 1$. The parameter $b$ is selected such that  \(\mu b >  \max\{1, \beta \cdot \frac{2(\delta-1)}{\delta}\)\}. Choose the block parameters
	$b_k, s_k$ according to \eqref{c1} with $m=m_k$ and confidence level $\gamma_k = 1/(k+1)^2$. Then, for almost every sample path 
		\(\omega\), there exists a constant \(A_3(\omega) > 0\) and an integer 
		\(K_r(\omega)\) such that \(\forall k \ge K_r(\omega)\),
	\begin{align}
			\Vert \mathbf{x}_k - x^* \Vert^2  
			\le A_3(\omega) \max\left\{ \frac{1}{k}, \left( \frac{\ln k}{k^{\beta}} \right)^{\frac{2(\delta-1)}{\delta}} \right\}. \label{eq:combined_thm}
	\end{align}
	 	In particular, when \(\delta = 2\) and \(\beta \ge 1\), we have $\forall k \ge K_r(\omega)$,
	\begin{align}
		\Vert \mathbf{x}_k - x^* \Vert^2  \le A_3 (\omega) \frac{\ln k }{k}. \label{eq:combined2_thm}
	\end{align}
	\end{theorem}
	
	\begin{proof}
	See Appendix \ref{app_theorem2}.
\end{proof}

\begin{remark}
The rate reduces to $\mathcal{O}(\frac{\ln k}{k})$ for the Gaussian noise $\delta=2$, which is worse than the standard rate $\mathcal{O}(\frac{1}{k})$. This is because the MoM estimator error bound in \eqref{part1} has a term $\ln(e^{1/8} \gamma^{-1})$. If $\gamma$ is set to be in a polynomial form as in the theorem, there will be a $\ln(k)$ term in the error. If $\gamma$ is set to be in an exponential form, the required samples per step will increase quite fast, which is not practical.
\end{remark}

\section{Improving MoM for Asymmetric Noise} \label{section_online}

The convergence analysis in the last section relies only on the noise conditions Assumptions \ref{ass:unbiased} and  \ref{ass:heavytail}, which
do not require the noise distribution to be symmetric. However, in practical
scenarios with asymmetric heavy-tailed noise, the MoM estimator
has a systematic bias that can slow down finite-sample convergence, because in this case the median does not equal to the mean. In this
section, we introduce the online bias correction method to handle this issue. 

\subsection{Algorithm Design by Online Bias Correction} \label{section_5}

At each iteration $k = 0,1,2,\ldots$, every player $i$ maintains a local action $x_{i,k}\in\Omega_i$.
Let $\mathbf{x}_k = [x_{1,k},\ldots,x_{N,k}]^{\top}\in\mathbb{R}^N$ be the action profile.

\begin{description}
	\item[Steps 1--3: ] 
	These steps are the same as in Section \ref{sec_al1}.
	In particular, draw $m_k$ i.i.d. samples, calculate $b_k$ and block size $s_k$,
	compute the stochastic gradients $g_{i,k}^j$ as in~\eqref{step1},
	form within-block averages $\bar{g}_{i,k}^{\ell}$ as in~\eqref{step2},
	and obtain the median-of-means estimate $\hat{g}_{i,k}$ as in~\eqref{step3}.
	
\item[Step 4: Online bias correction.]
Let $\eta_k\in[0,1]$ be a decaying confidence coefficient such that $\lim_{k\to\infty}\eta_k = 0$.
The corrected gradient estimate is formed by 
\begin{align}
	\tilde{g}_{i,k} = (1-\eta_k)\hat{g}_{i,k} + \eta_k \bar{g}_{i,k}, \label{eq:onlinecorr}
\end{align}
where $\bar{g}_{i,k}$ is the sample mean
	\begin{equation}
	\bar{g}_{i,k}  = \frac{1}{b_k}\sum_{\ell=1}^{b_k} \bar{g}_{i,k}^\ell.\label{step4_1}
\end{equation}

\item[Step 5: Action update.]
Perform a projected gradient step using the corrected gradient:
\begin{align}
	x_{i,k+1} = \mathbb{P}_{\Omega_i}\bigl[x_{i,k} - \alpha_k\tilde{g}_{i,k}\bigr], \label{eq:update_online}
\end{align}
where $\alpha_k>0$ is the step-size.
\end{description}

	\subsection{Convergence Theorem for Online Bias Correction}
	
	We state and prove the almost sure convergence of the algorithm with online bias correction. 
	
	\begin{theorem}[Almost Sure Convergence for MoM with Online Bias Correction]
		\label{thm:online}
		Let Assumptions  \ref{boundedness}--\ref{ass:heavytail} hold.
		Consider the algorithm in Section \ref{section_5} with step-size
		$\alpha_k = b(k+1)^{-a}$ for some $a\in(0,1]$ and $b>0$,
		sample size $m_k = c\lceil (k+1)^{\beta}\rceil$  for some $\beta>\frac{1}{ \delta-1}$ and integer $c\geq 1$. Choose the block parameters
		$b_k, s_k$ according to \eqref{c1} with $m=m_k$ and confidence level $\gamma_k = 1/(k+1)^2$. 
	Let the decay coefficient $\eta_k = \eta_0(k+1)^{-\rho}$ ($\eta_0>0,\rho>0$). 
		Then the sequence $\{\mathbf{x}_k\}$ generated by the algorithm converges almost surely to the unique Nash equilibrium $x^*$, i.e.
		\begin{align}
		\lim_{k\to\infty}\Vert\mathbf{x}_k - x^*\Vert = 0 \quad \text{a.s.}
		\end{align}
	\end{theorem}
	
\begin{proof}
See Appendix \ref{appd_V1}.
\end{proof}

\begin{remark}
	Compared with Theorem \ref{thm:convergence}, this theorem requires an additional condition on $\beta$, i.e., $\beta>\frac{1}{ \delta-1}$. The online correction step employs the sample mean $\bar g_{i,k}$. 
	To guarantee almost sure convergence, the tail probability of the sample mean error must be summable, leading to the requirement $\beta(\delta-1) > 1$.
	This condition ensures that the $\delta$-th moment of $\bar\epsilon_{i,k}$ decays sufficiently fast, so that the sample mean error is eventually negligible with probability one. In practice, if the noise distribution is known to be symmetric, we shall use the method and conclusions given in Section \ref{sec_mom}, since it provides a milder condition for the parameters. The parameter $\beta$ represents the growth rate of the required samples per step, and thus its selection is quite important. 
\end{remark}

\subsection{Convergence Rate for Online Bias Correction}

The convergence rate of the algorithm with online bias correction is given in the following theorem. 

\begin{theorem}[Almost Sure Convergence Rate for MoM with Online Bias Correction]
	\label{thm:rate_online}
	Let Assumptions \ref{boundedness}--\ref{ass:heavytail} hold.
	Consider the algorithm in Section \ref{section_5} with with step-size $\alpha_k = b(k+1)^{-1}$ for some $b>0$,
	sample size $m_k = c\lceil (k+1)^\beta \rceil$ with $\beta > \frac{1}{\delta-1}$ and integer $c\ge 1$.
	Choose the block parameters $b_k, s_k$ according to \eqref{c1} with $m=m_k$,
	confidence level $\gamma_k = 1/(k+1)^2$,
	and decay coefficient $\eta_k = \eta_0 (k+1)^{-\rho}$ with $\eta_0>0$, $\rho>0$.
	Select $\zeta \in \bigl(0, \frac{\beta(\delta-1)-1}{\delta}\bigr)$ and choose $b$ such that
	$\mu b > \max\bigl\{ 1, \beta\frac{2(\delta-1)}{\delta}, 2\rho+2\zeta \bigr\}$.
	Then, for almost every sample path $\omega$, there exist a constant $B(\omega)>0$ and an integer $\kappa_r(\omega)$ such that $\forall k \ge \kappa_r(\omega)$,
	\begin{align}
		\Vert \mathbf{x}_k - x^* \Vert^2 \le B(\omega) \max\Bigl\{ \frac{1}{k}, \Bigl(\frac{\ln k}{k^\beta}\Bigr)^{\frac{2(\delta-1)}{\delta}}, \frac{1}{k^{2\rho+2\zeta}} \Bigr\}. \label{eq:rate_online}
	\end{align}
	In particular, when $\delta=2$, $\beta\ge 1$, and $2\rho+2\zeta \ge 1$, we have $\forall k \ge \kappa_r(\omega)$,
	\begin{align}
		\Vert \mathbf{x}_k - x^* \Vert^2 \le B(\omega) \frac{\ln k}{k}. \label{eq:rate_online_special}
	\end{align}
\end{theorem}

\begin{proof}
See Appendix \ref{appd_th4}.
\end{proof}

\begin{remark}
	Although the online bias‑corrected MoM estimator does not provide a better convergence rate than
	as the plain MoM estimator in theory, it is specifically designed to mitigate the systematic bias that arises
	from asymmetric heavy‑tailed noise in finite‑sample regimes.
	Both methods are guaranteed to converge to the true Nash equilibrium.
	The correction is recommended when the noise distribution is severely skewed, as it can accelerate
	practical convergence by reducing the bias without requiring additional assumptions.
\end{remark}

\section{Simulation}
\label{Simulation}

We consider a smart grid demand response problem \cite{ye2016game}. In this problem, several electricity users decide how much power to use. The price of electricity is not fixed, and it increases when the total consumption of all users is high and decreases when it is low. Each user's cost depends not only on its own consumption, but also on the total consumption of all users through the price. Each user expects to minimize its own cost. 

We consider \(N=10\) users. The objective function of user \(i\) is
\begin{equation}
	f_i(x,\xi_i) = a_i (x_i - \hat{x}_i)^2 + \left(b\sum_{j=1}^{10}x_j + p_0 + \xi_i\right)x_i,
\end{equation}
where \(x_i\) is the energy consumption (kWh), \(\hat{x}_i\) is the reference consumption,  \(a_i\), $b$ and $p_0$ are positive coefficients and $ \xi_i$ is a random variable.  The constraint set is  \([0,150]\). The game parameters are listed in Table I.

\begin{table*}[htbp]
	\centering
	\caption{Game parameters.}
	\begin{tabular}{c|c}
		\hline
		Parameter & Value \\
		\hline
		Action bound & \([0, 150]\) kWh \\
		Reference consumption \(\hat{x}_i\) & \([78, 68, 54, 47, 60,\) \\
		& \(78, 58, 62, 50, 50]\) \\
		Coefficient \(a_i\) & \([1.1465, 1.0404, 1.0832, 0.8082, 1.1880,\) \\
		& \(1.1330, 0.8849, 0.8727, 0.8734, 0.9217]\) \\
		Price sensitivity \(b\) & 0.04 \\
		Base price \(p_0\) & 5 \\
		\hline
	\end{tabular}
\end{table*}

The cost of user \(i\) has two parts. The first part is a penalty \(a_i (x_i - \hat{x}_i)^2\). Here, \(\hat{x}_i\) is the user's normal consumption. The user gets a penalty for using more than \(\hat{x}_i\), because it costs extra energy. The user also gets a penalty for using less than \(\hat{x}_i\), because it reduces comfort. The second part is the payment for the electricity, which is the price times the consumption \(x_i\). Since the price depends on the total consumption, each user's decision affects the costs of other users. A Nash equilibrium is a consumption profile where no user can lower its own cost by changing its consumption alone, given that all other users keep their consumption fixed. 

The price is affected by random fluctuations, modeled by a random variable $ \xi_i$.   These include renewable energy variations, sudden changes in demand, and even international situations such as energy supply disruptions. These disturbances often cause large price jumps. Because of this, the noise in the price may be not in a  Gaussian form that has a good property. We model this noise using a heavy-tailed Pareto distribution. 

 The pseudo-gradient is strongly monotone with modulus \(\mu = 1.67\) and Lipschitz continuous. The unique Nash equilibrium  of this game is about $
	x^* = [66.4,55.4, 42.2,31.4,49.1, 66.3,
	43.5,47.2,35.5,36.2]^\top.
$

\subsection{Symmetric Heavy-Tailed Noise}

The noise $\xi_i$ is zero-mean and follows a symmetrized Pareto distribution with tail index $\alpha$. Specifically, we let
$\xi_i = S Z,$
where $S$ is a Rademacher random variable satisfying
$\mathbb{P}(S=1)=\mathbb{P}(S=-1)=\frac{1}{2},$
and $Z$ is independent of $S$ and follows a Pareto distribution with density
$p(z)=\alpha z^{-\alpha-1}, \qquad z\ge 1.$
Since $\mathbb{E}[S]=0$, we have
$\mathbb{E}[\xi_i]=\mathbb{E}[S]\mathbb{E}[Z]=0.$
Moreover, for any $0<\delta<\alpha$,
$\mathbb{E}[|\xi_i|^\delta]
=
\mathbb{E}[Z^\delta]
<\infty.$
Therefore, when $\alpha>1$, the noise has a finite $\delta$-th moment for some $1<\delta<\alpha$, which satisfies the assumptions. 

We compare five methods:
\begin{enumerate}
	\item \textbf{Gradient Clipping} \cite{sun2025distributed1}: $x_{i,k+1} = \mathbb{P}_{\Omega_i}[x_{i,k} - \alpha_k \,\mathrm{clip}(\nabla_{x_i}f_i(\mathbf{x}_k,\xi_{i,k}),\tau_k)]$, using one sample per iteration.
	\item \textbf{Clipped-SGDA} \cite{sadiev2023high}: $x_{k+1} = \mathbb{P}_{\Omega}[x_k - \gamma  \,\mathrm{clip}(F_{\xi_k}(x_k),\tau_k)]$, which is equivalent to applying clipping directly to the whole pseudo-gradient vector.
	\item \textbf{Clipped-SEG} \cite{sadiev2023high}: one extragradient step with two independent samples per iteration.
	\item \textbf{MoM} (this work): $m_k$ samples are divided into $b_k$ blocks of size $s_k$, the block means are computed, and their median is used as the gradient estimate.
	\item \textbf{MoM with fixed $m$} (a practical modification of the method in this work): In practice, the MoM method usually uses a fixed number of samples instead of that of going to infinity. Here we fix the number of samples in each iteration to $m=20$.
	
\end{enumerate}

For the methods in   \cite{sun2025distributed1} and this work, we use the same step size $\alpha_k = 1/(k+1)$. For the methods in \cite{sadiev2023high}, we take $\gamma=0.005$.  For the clipping-based methods, we use the same clipping threshold $\tau_k = 20 (k+1)^{0.2}$.  
The MoM method with growing $m$ uses $m_k = k+1$ samples per iteration, which satisfies Theorem \ref{thm:convergence}. 

Because MoM applies more samples at one iteration, 
we  make a fair comparison based on sample complexity. All runs stop after each player consumes $100,000$ local stochastic gradient samples. The horizontal axis in the figures represents the cumulative number of samples, which means that the MoM indeed has a less number of iteration.

We conducted 20 Monte Carlo trials and take the mean of the results. Fig. \ref{fig:sym-heavy} shows the performance for Pareto tail index $\alpha=1.8$.  Fig. \ref{fig:sym-heavy}(a) shows the experiment results with the $x$-axis representing the number of consumed samples.  
Fig. \ref{fig:sym-heavy}(b) shows the evolution of the error with the $x$-axis representing the iterations.  It can be seen that all five algorithms converge and the proposed methods have a better performance for both scales. Due to the total sample budget constraint, MoM completes only about 400 updates on the iteration count axis. Its magnitude is already significantly lower than those of the compared algorithms.

\begin{figure} [htbp]
	\centering
	\begin{subfigure}{0.4\textwidth}
		\centering
		\includegraphics[width=\textwidth]{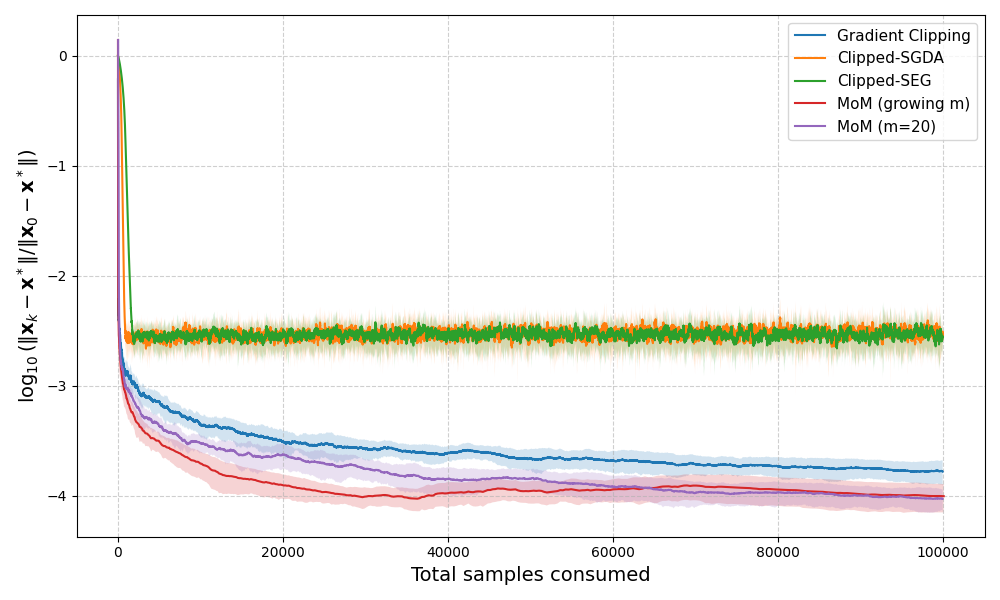}
		\caption{Relative error vs.~total number of consumed samples.}
		\label{fig:sym-tail18}
	\end{subfigure}
	\hfill
	\begin{subfigure}{0.4\textwidth}
		\centering
		\includegraphics[width=\textwidth]{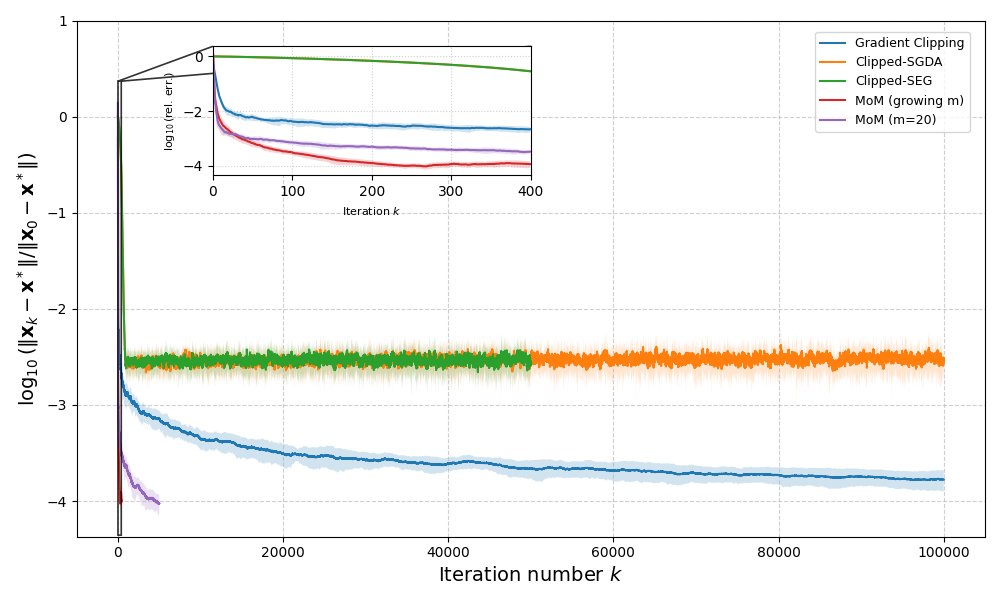}
		\caption{Relative error vs.~iterations.}
		\label{fig:sym-tail12}
	\end{subfigure}
	\caption{Convergence under symmetric heavy‑tailed noise with $\alpha=1.8$.}
	\label{fig:sym-heavy}
\end{figure}

To further evaluate the robustness of the algorithms under different tail heaviness, we record the final relative error after exhausting the $100,000$ sample budget in Table~\ref{fig:attack_comparison}. It shows that the MoM methods perform better than the clipping methods for all tail indices.

We also compare the five algorithms under two different symmetric heavy-tailed distributions, i.e., Student-$t$ with $1.5$ degrees of freedom and symmetric $\alpha$-stable with $\alpha=1.5$. The final relative errors are presented in Table~\ref{tab:dist-error}. It can be seen that the MoM methods achieve lower mean relative errors for these distributions.

\begin{table*}[htbp]
	\centering
	\caption{Final relative error (mean $\pm$ std) for different tail indices $\alpha$ after using  $100000$ samples.}
	\label{tab:tail-error}
	\small
	\begin{tabular}{c|ccccc}
		\hline
		$\alpha$ & Gradient Clipping & Clipped-SGDA & Clipped-SEG & MoM (growing $m$) & MoM ($m=20$) \\
		\hline
		2.0 & 1.31e-04±2.57e-05 & 2.22e-03±7.01e-04 & 1.98e-03±5.01e-04 & 9.76e-05±2.04e-05 & 9.60e-05±2.21e-05 \\
		1.5 & 2.46e-04±6.60e-05  &   4.74e-03±2.11e-03  &   4.49e-03±1.14e-03   &  1.14e-04±1.54e-05  &   1.17e-04±2.72e-05 \\
		1.2 & 4.57e-04±1.13e-04  &  8.48e-03±2.20e-03 &   7.92e-03±1.90e-03   &  1.25e-04±2.42e-05  &   1.46e-04±3.51e-05   \\      
		\hline
	\end{tabular}
\end{table*}

\begin{table*}[htbp]
	\centering
	\caption{Final relative error (mean $\pm$ std) for different symmetric heavy-tailed distributions after using  $100000$ samples.}
	\label{tab:dist-error}
	\small
	\begin{tabular}{c|ccccc}
		\hline
		Distribution & Gradient Clipping & Clipped-SGDA & Clipped-SEG & MoM (growing $m$) & MoM ($m=20$) \\
		\hline
		Student-$t$($df=1.5$)  &   2.27e-04±4.14e-05  &  3.81e-03±1.34e-03 &   4.26e-03±1.42e-03 &    6.65e-05±1.25e-05  &  7.34e-05±1.25e-05 \\
		Stable($\alpha=1.5$)  &  1.59e-04±3.94e-05 &   3.32e-03±1.46e-03  &  3.17e-03±1.05e-03  &  7.30e-05±1.58e-05  &  6.68e-05±1.27e-05  \\
		\hline
	\end{tabular}
\end{table*}

\subsection{Algorithm Performance under Gradient Attacks}

To further evaluate the robustness against malicious gradient attacks, we adopt a gradient-adaptive attack. For each sample, with probability $p$, the gradient is added by a large constant multiplied by the opposite sign of the true gradient, i.e., $-\mathrm{sign}(\nabla_{x_i}J_i(\mathbf{x}_k)) \times 100$. This kind of attack is more  destructive than random attacks since the gradient has an opposite direction.

We test two attack probabilities, $p=0.1$ and $p=0.3$, and compare the convergence behaviour of all five
algorithms under the same symmetric Pareto noise with tail index $1.8$.
Fig. \ref{fig:attack_comparison}  shows the relative error versus total consumed samples for the two attack probabilities. It can be seen that the MoM-based methods achieve a faster convergence rate, while the clipping-based methods are no longer effective as the case shown in Fig. 	\ref{fig:sym-heavy}. When the attack probability increases to $0.3$, the MoM methods lose their  advantage and
perform similarly to the clipping methods, since the probability of the corrupted gradients is too large.

\begin{figure}[htbp]
	\centering
	\begin{subfigure}{0.4\textwidth}
		\centering
		\includegraphics[width=\textwidth]{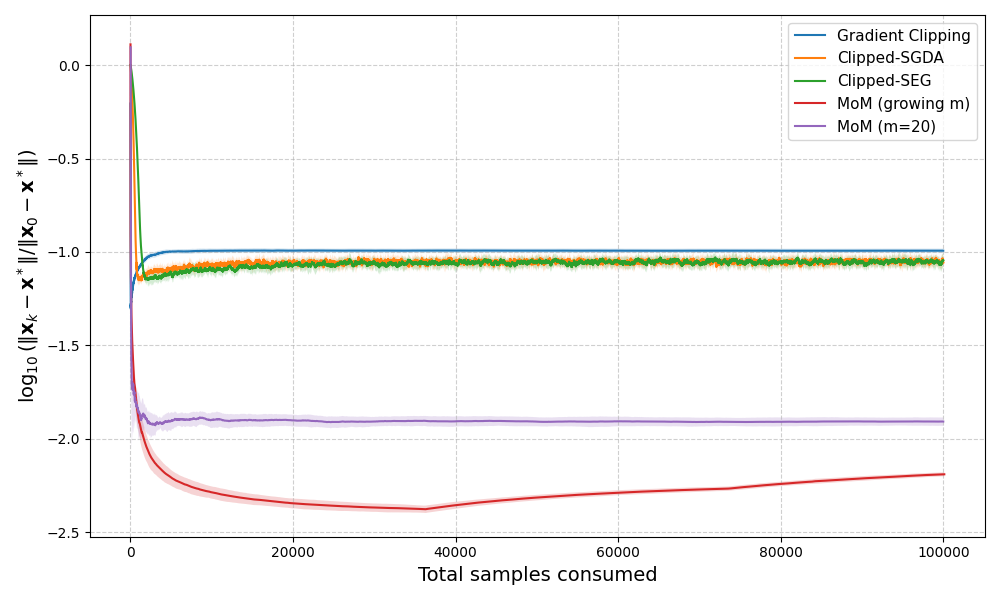}
		\caption{$p=0.1$}
		\label{fig:attack01}
	\end{subfigure}
	\hfill
	\\
	\begin{subfigure}{0.4\textwidth}
		\centering
		\includegraphics[width=\textwidth]{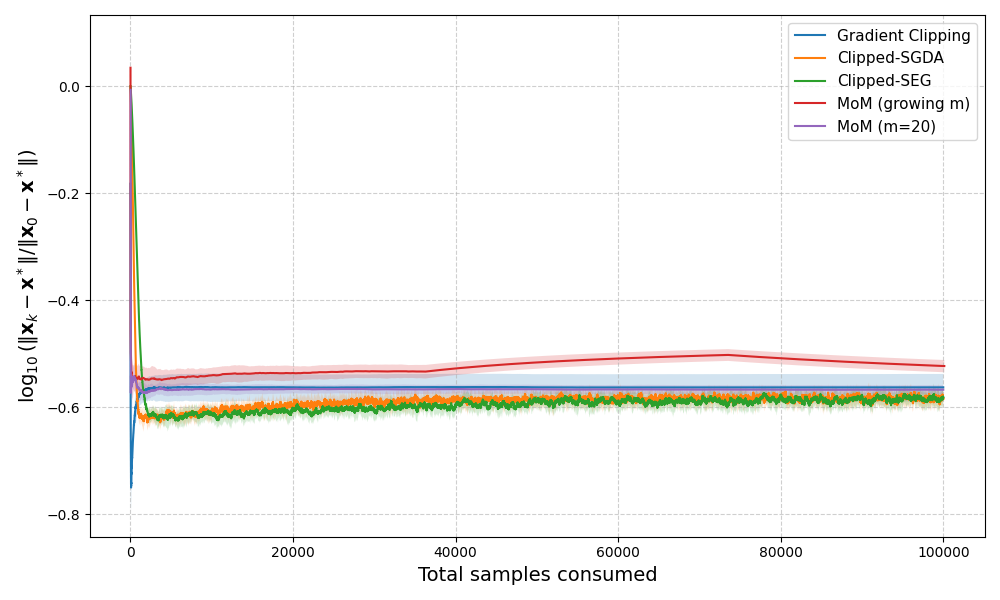}
		\caption{$p=0.3$}
		\label{fig:attack04_samples}
	\end{subfigure}
	\hfill
	\caption{Convergence under gradient attacks.}
	\label{fig:attack_comparison}
\end{figure}

\subsection{Asymmetric Heavy-Tailed Noise}

We now consider asymmetric heavy‑tailed noise.
The random variable $\xi_i$ follows a shifted Pareto distribution with tail index $\alpha$.
Let $Z\sim\mathrm{Pareto}(\alpha)$, i.e., $p_Z(z)=\alpha z^{-(\alpha+1)}$ for $z\ge 1$, and $\mathbb{E}[Z]=\frac{\alpha}{\alpha-1}$.
Define $\xi_i = Z - \mathbb{E}[Z]$, which preserves the right skewness of the Pareto law while ensuring zero mean.

We add the MoM with online bias correction for comparison, where $\eta_k = \,(k+1)^{-0.2}$. For the methods in  \cite{sun2025distributed1} and this work, we use the same step size $\alpha_k = 2/(k+1)$, where we increases the coefficients to satisfy the conditions in the theorems. 
For the methods in  \cite{sadiev2023high}, we take $\gamma=0.005$. 
All the algorithms that employ clipping share the same threshold $\tau_k = 20 (k+1)^{0.2}$.  Theorem \ref{thm:online} requires $\beta>\frac{1}{ \delta-1}$. In the simulation, for  $\alpha=1.5$, we take  $\beta=3$. We also conducted the simulation for a fixed $m$ even though it does not satisfy the conditions in Theorem \ref{thm:online}.

It can be seen from Fig. \ref{fig:asym-heavy} that the proposed correction strategy works and all corrected algorithms perform better than its plain version.

\begin{figure}[htbp]
		\centering
		\includegraphics[width=0.4\textwidth]{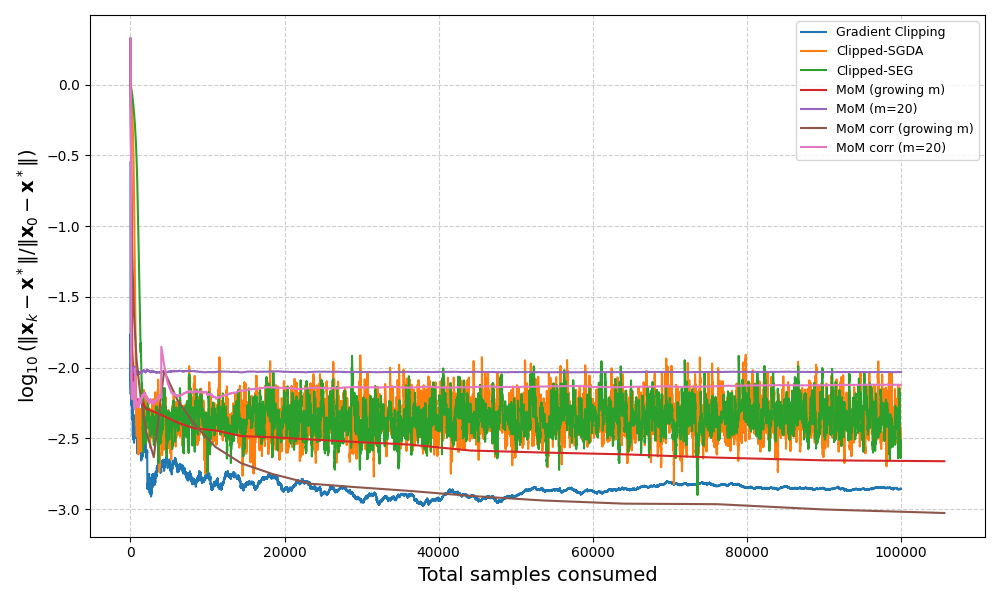}
		\label{fig:asym-tail18-beta21}
	\caption{Convergence under asymmetric heavy‑tailed noise with $\alpha=1.5$.}
	\label{fig:asym-heavy}
\end{figure}

\section{Conclusions}\label{S7}
This work investigated the Nash equilibrium seeking problem in non-cooperative games under heavy-tailed noise, where only a finite $\delta$-th moment with $1<\delta\leq2$ is required. A Nash equilibrium seeking algorithm based on the MoM robust estimator was proposed, which eliminates the need for a clipping threshold and naturally withstands adversarial gradient attacks. Under the strong monotonicity assumption, the algorithm was shown to converge almost surely to the unique Nash equilibrium, and an almost sure convergence rate was derived. To address the systematic bias arising from asymmetric noise distributions, an online bias correction strategy was further introduced, and its convergence guarantees were established. Simulation results show the effectiveness and efficiency of the proposed methods. Future research will focus on extending the framework to generalized Nash equilibrium problems and considering partial-decision information settings.

\bibliographystyle{IEEEtran}
\bibliography{bib}
\vspace{1cm}

\appendix
\subsection{Proof of inequality \eqref{part2} in Lemma \ref{lemmaz}} \label{proof_lemmaz}

Define $W_j=-Z_j$ for $j=1,\cdots, m$. Then, $W_1,\cdots, W_m$ are i.i.d. with mean $\bar{\theta}_W=-\bar{\theta}$ and satisfy the same moment condition 
	\begin{align}
	\mathbb{E}[\vert W_1-\bar{\theta}_W\vert^{1+\alpha}] \leq u.
\end{align}

Let $\bar{W}_\ell:=\frac{1}{s}\sum_{j\in B_\ell} W_j$ and $\hat{\theta}_W:= \text{median}(\bar{W}_1,\cdots,\bar{W}_b)$. By the property of the median under sign reversal,
	\begin{align}
	\hat{\theta}_W=\text{median}(-\bar{Z}_1,\cdots,-\bar{Z}_b)=-\hat{\theta}.
\end{align}

Inequality \eqref{part1} applied to $W$ yields, with probability at least $1-\gamma$,
	\begin{align}
	\hat{\theta}_W \leq \bar{\theta}_W+(12v)^{\frac{1}{1+\alpha}}\left(\frac{16\text{ln}(e^{1/8} \gamma^{-1})}{m}\right)^{\frac{\alpha}{1+\alpha}}, \label{parta}
\end{align}
which implies that 
	\begin{align}
	-\hat{\theta} \leq -\bar{\theta}+(12v)^{\frac{1}{1+\alpha}}\left(\frac{16\text{ln}(e^{1/8} \gamma^{-1})}{m}\right)^{\frac{\alpha}{1+\alpha}}. \label{partb}
\end{align}

\subsection{Proof of Lemma \ref{lem:cond_tail_bound}} \label{conditional}

Fix $k\ge0$ and $i$. The history $\mathcal{F}_k$ contains all randomness before step $k$,
so $\mathbf{x}_k$ is $\mathcal{F}_k$-measurable. Given $\mathcal{F}_k$, the per-sample
noises $\epsilon_{i,k}^j = g_{i,k}^j - \nabla_{x_i}J_i(\mathbf{x}_k)$ for $j=1,\dots,m_k$
are conditionally independent and identically distributed. This follows because
the samples $\xi_{i,k}^j$ are drawn independently from the same distribution.

By Assumption \ref{ass:unbiased}, $\mathbb{E}[\epsilon_{i,k}^j \mid \mathcal{F}_k] = 0$ a.s., and by
Assumption \ref{ass:heavytail}, $\mathbb{E}[|\epsilon_{i,k}^j|^\delta \mid \mathcal{F}_k] \le \nu^\delta$ a.s.
for some $\delta\in(1,2]$ and $\nu>0$. Let $\Omega_0$ be the set of all $\omega$ such that the conditional mean and moment
bounds hold at step $k$. By Assumptions \ref{ass:unbiased} and \ref{ass:heavytail}, $\mathbb{P}(\Omega_0)=1$.
Fix any $\omega\in\Omega_0$. In the conditional probability space given
$\mathcal{F}_k(\omega)$, the variables $Z_j = g_{i,k}^j$ are i.i.d.\ with mean
$\bar{\theta} = \nabla_{x_i}J_i(\mathbf{x}_k(\omega))$ and satisfy
\begin{align}
\mathbb{E}\bigl[ | g_{i,k}^j - \nabla_{x_i}J_i(\mathbf{x}_k(\omega))|^\delta \mid \mathcal{F}_k \bigr](\omega) \le \nu^\delta.
\end{align}
Set $\alpha = \delta-1$ and $u = \nu^\delta$. The conditions of Lemma \ref{lemmaz} are met. Applying the lemma gives
\begin{align}
\mathbb{P}\bigl( |\hat{g}_{i,k} - \nabla_{x_i}J_i(\mathbf{x}_k(\omega))| > \mathcal{E}_k(\gamma) \mid \mathcal{F}_k \bigr)(\omega) \le 2 \gamma.
\end{align}

Since the inequality holds for every $\omega\in\Omega_0$ and $\mathbb{P}(\Omega_0)=1$,
we obtain
\begin{align}
\mathbb{P}\bigl( |\hat{g}_{i,k} - \nabla_{x_i}J_i(\mathbf{x}_k)| > \mathcal{E}_k(\gamma) \mid \mathcal{F}_k \bigr) \le 2 \gamma \quad \text{a.s.}
\end{align}
This completes the proof.

\subsection{Proof of Theorem \ref{thm:convergence}} \label{proof_t1}
Let $\epsilon_{i,k}^{\mathrm{MoM}} = \hat{g}_{i,k} - \nabla_{x_i}J_i(\mathbf{x}_k)$
and $\epsilon_k^{\mathrm{MoM}} = [\epsilon_{1,k}^{\mathrm{MoM}},\dots,\epsilon_{N,k}^{\mathrm{MoM}}]^\top$.

According to the conditions in the theorem, there exists a $K_a$ such that $m_k\geq 16\ln(e^{1/8}\gamma_k^{-1})+2$ for any $k\geq K_a$.

By Lemma \ref{lem:cond_tail_bound},
\begin{align}
\mathbb{P}\bigl( |\epsilon_{i,k}^{\mathrm{MoM}}| > \mathcal{E}_k(\gamma_k) \mid \mathcal{F}_k \bigr) \le 2 \gamma_k \quad \text{a.s.}
\end{align}
Taking the expectation gives $\mathbb{P}(|\epsilon_{i,k}^{\mathrm{MoM}}| > \mathcal{E}_k(\gamma_k)) \le 2\gamma_k.$ 
Let $A_{i,k} = \{ |\epsilon_{i,k}^{\mathrm{MoM}}| > \mathcal{E}_k(\gamma_k) \}$.
Since $\sum_{k=0}^\infty\mathbb{P}(A_{i,k}) \leq 2\sum_{k=0}^\infty \gamma_k< \infty$, the Borel--Cantelli lemma implies
$\mathbb{P}(\limsup_{k\rightarrow \infty} A_{i,k}) = 0$. Thus, for almost every sample path $\omega$,
there exists $K_a\leq K_i(\omega) < \infty$ such that for all $k \ge K_i(\omega)$,
\begin{align}
|\epsilon_{i,k}^{\mathrm{MoM}}(\omega)| \le \mathcal{E}_k(\gamma_k).
\end{align}
Set $K_b(\omega) = \max_{1\le i\le N} K_i(\omega)$. Then for all $k \ge K_b(\omega)$,
\begin{align}
\Vert \epsilon_k^{\mathrm{MoM}}(\omega)\Vert \le \sqrt{N} \mathcal{E}_k(\gamma_k).
\end{align}

For large $k$, we can simplify the bound on $\mathcal{E}_k(\gamma_k)$.
According to \eqref{13}, letting 
$\gamma_k = 1/(k+1)^2$ and $m_k = c \lceil (k+1)^\beta \rceil$ with $\beta>0$ gives
\begin{align}
\mathcal{E}_k(\gamma_k) = C_1 \left( \frac{16\ln\bigl(e^{1/8}(k+1)^2\bigr)}{m_k} \right)^{\frac{\delta-1}{\delta}}.
\end{align}
Using $\ln(e^{1/8}(k+1)^2) = \frac{1}{8} + 2\ln(k+1) \le 3\ln(k+1)$ for all
$k\geq 1$, and noting $m_k \ge (k+1)^\beta$, we obtain
\begin{align}
\frac{16\ln(e^{1/8}\gamma_k^{-1})}{m_k} \le \frac{48\ln(k+1)}{(k+1)^\beta} \le 48 \frac{\ln(k+1)}{k^\beta}.
\end{align}
Therefore, for all  $k\geq 1$,
\begin{align}
\mathcal{E}_k(\gamma_k) \le& C_1 \left( 48 \frac{\ln(k+1)}{k^\beta} \right)^{\frac{\delta-1}{\delta}}
\notag\\
= &C_2 \left( \frac{\ln(k+1)}{k^\beta} \right)^{\frac{\delta-1}{\delta}}:=\mathcal{E}_k, \label{kesi1}
\end{align}
where $C_2 : = C_1 \cdot 48^{(\delta-1)/\delta}$ is a finite constant depending only on
$\delta$ and $\nu$.

Let  
\begin{align}
\bar{\eta}_k :=\sqrt{N} \mathcal{E}_k . \label{eta}
\end{align}
Hence, for almost every path,
\begin{equation}
	\Vert \epsilon_k^{\mathrm{MoM}}\Vert  \le \bar{\eta}_k \quad \text{for all } k \ge K_c(\omega):=\max\{1,K_b(\omega)\}.
	\label{eq:det_bound}
\end{equation}

Let $V_k: = \Vert \mathbf{x}_k - x^*\Vert ^2$. By the non-expansiveness of the projection,
\begin{align}
\Vert x_{i,k+1} - x_i^*\Vert ^2 \le \Vert x_{i,k} - \alpha_k \hat{g}_{i,k} - x_i^*\Vert ^2.
\end{align}

Expanding and summing over $i=1,\dots,N$ yields
\begin{align}
	V_{k+1} \le V_k - 2\alpha_k \hat{g}_k^\top (\mathbf{x}_k - x^*) + \alpha_k^2 \Vert \hat{g}_k\Vert ^2,
\end{align}
where $\hat{g}_k:=[\hat{g}_{1,k},\cdots,\hat{g}_{N,k}]^\top$.

Write $\hat{g}_k = F(\mathbf{x}_k) + \epsilon_k^{\mathrm{MoM}}$. Then
\begin{align}
	V_{k+1} \le& V_k - 2\alpha_k F(\mathbf{x}_k)^\top (\mathbf{x}_k - x^*) \notag\\
	&- 2\alpha_k (\epsilon_k^{\mathrm{MoM}})^\top (\mathbf{x}_k - x^*) + \alpha_k^2 \Vert \hat{g}_k\Vert ^2.\label{V2}
\end{align}

According to the strong monotonicity in Assumption \ref{strongly} and the equilibrium condition, we have
\begin{align}
	F(\mathbf{x}_k)^\top (\mathbf{x}_k - x^*) \ge \mu V_k. \label{1}
\end{align}

For the cross term, using the Young's inequality gives
\begin{align}
	-2(\epsilon_k^{\mathrm{MoM}})^\top (\mathbf{x}_k - x^*) \le& 2 \Vert\epsilon_k^{\mathrm{MoM}}\Vert \sqrt{V_k} \notag\\
	\le &\mu V_k + \frac{1}{\mu} \Vert\epsilon_k^{\mathrm{MoM}}\Vert^2. \label{2}
\end{align}

By Assumption \ref{Lipschitz} and the compactness of $\Omega$ in Assumption \ref{boundedness}, there exists $G>0$ such that $\Vert F(\mathbf{x})\Vert  \le G$ for all $\mathbf{x}\in\Omega$. Thus
\begin{align}
	\Vert \hat{g}_k\Vert ^2 \le 2G^2 + 2\Vert \epsilon_k^{\mathrm{MoM}}\Vert ^2.\label{3}
\end{align}

Substituting \eqref{1}--\eqref{3} into \eqref{V2} gives
\begin{align}
	V_{k+1} \le &(1 - \mu\alpha_k) V_k + \alpha_k \Bigl( \frac{1}{\mu} + 2\alpha_k \Bigr) \Vert \epsilon_k^{\mathrm{MoM}}\Vert ^2 \notag\\
	&+ 2G^2\alpha_k^2.
\end{align}

Since $\alpha_k \to 0$,  for $k\geq K_d$, $\frac{1}{\mu} + 2\alpha_k \le C_3$ for some $C_3>0$ and $K_d>0$. Hence, for all $k\geq K_d$,
\begin{equation}
	V_{k+1} \le (1 - \mu\alpha_k) V_k + C_3\alpha_k \Vert \epsilon_k^{\mathrm{MoM}}\Vert ^2 + 2G^2\alpha_k^2.
	\label{eq:V_rec}
\end{equation}

Fix a sample path $\omega$ for which \eqref{eq:det_bound} holds. For $k \ge K_c(\omega)$,
we have $\Vert \epsilon_k^{\mathrm{MoM}}\Vert  \le \bar{\eta}_k$. Substituting into \eqref{eq:V_rec} gives
\begin{align}
V_{k+1} \le (1 - \mu\alpha_k) V_k + \beta_k, \quad \forall k \ge \max\{ K_c(\omega),K_d\}, \label{eq1c}
\end{align}
where $\beta_k := C_3\alpha_k \bar{\eta}_k^2 + 2G^2\alpha_k^2$.

Let $\alpha_k = b/(k+1)^a$ with $a\in(0,1]$. Then $\sum_{k=0}^\infty \alpha_k = \infty$ and $\alpha_k \to 0$. There exists $K_e$ such that $0<\mu\alpha_k\leq1$ for $k\geq K_e$.
Moreover,
\begin{align}
	\frac{\beta_k}{\alpha_k} = &C_3 \bar{\eta}_k^2 + 2G^2\alpha_k \notag\\
	= &C_3 NC_2^2 \Bigl( \frac{\ln(k+1)}{k^\beta} \Bigr)^{\frac{2(\delta-1)}{\delta}} + \frac{2G^2b}{(k+1)^a} \to 0.
\end{align}

All conditions of Lemma \ref{a} are satisfied. Therefore, $V_k \to 0$ on this sample path. Since such paths have probabiliy one,
\begin{align}
\lim_{k\to\infty} \Vert \mathbf{x}_k - x^*\Vert  = 0 \quad \text{a.s.}
\end{align}

\subsection{Proof of Lemma \ref{lem:chung-log}} \label{appd_1}

	For $k\ge 2$ define
\begin{align}
	U_k=A\frac{(\ln k)^\tau}{k^p}.
\end{align}
 For $k=1$ we may set $U_1=1$ arbitrarily.

Thus, we have
\begin{align}
	\frac{U_{k+1}}{U_k}=\Big(\frac{k}{k+1}\Big)^p\Big(\frac{\ln(k+1)}{\ln k}\Big)^\tau.\label{u1}
\end{align}

According to the Taylor expansion, we have
\begin{align}
	\Big(1+\frac{1}{k}\Big)^{-p}&=1-\frac{p}{k}+\frac{p(p+1)}{2k^2}+O(k^{-3}),\notag\\
	\ln(k+1)&=\ln(k)+\ln(1+\frac{1}{k})\notag\\
	&=\ln k+\frac{1}{k}-\frac{1}{2k^2}+O(k^{-3}),\label{u2}
\end{align}
by which we obtain
\begin{align}
	\frac{\ln(k+1)}{\ln k}=1+\frac{1}{k\ln k}-\frac{1}{2k^2\ln k}+O\Big(\frac{1}{k^3\ln k}\Big).
\end{align}

Let $y:=\frac{\ln(k+1)}{\ln k}-1=\frac{1}{k\ln k}-\frac{1}{2k^2\ln k}+O\Big(\frac{1}{k^3\ln k}\Big)=\frac{1}{k\ln k}+O(k^{-2}(\ln k)^{-1})=O(\frac{1}{k\ln(k)})$. When $k\rightarrow\infty$, $y\rightarrow 0$. Thus, by the Taylor expansion, we have
\begin{align}
	(1+y)^\tau=1+\tau y+\frac{\tau(\tau-1)}{2}y^2+O(y^3). \label{y}
\end{align}
Substituting $y$ into \eqref{y} gives
\begin{align}
	(1+y)^\tau=&1+\tau\frac{1}{k\ln k}+\tau O(k^{-2}(\ln k)^{-1}) \notag\\
	& +\frac{\tau(\tau-1)}{2}O(\frac{1}{k^2(\ln(k))^2})+O(\frac{1}{k^3(\ln(k))^3}) \notag\\
	=&1+\frac{\tau}{k\ln k}+O\Big(\frac{1}{k^2\ln k}\Big).
\end{align}
i.e., 
\begin{align}
	\Big(\frac{\ln(k+1)}{\ln k}\Big)^\tau=1+\frac{\tau}{k\ln k}+O\Big(\frac{1}{k^2\ln k}\Big).\label{u3}
\end{align}

Based on \eqref{u1}, \eqref{u2} and \eqref{u3}, we have
\begin{align}
	\frac{U_{k+1}}{U_k}=1-\frac{p}{k}+\frac{\tau}{k\ln k}+s_k,
\end{align}
where the remainder $s_k$ satisfies $|s_k|\le C_4/k^2$ for all $k\ge K_f$ with suitable constants $C_4$ and $K_f$ which depend on $p$ and $\tau$ only.

From the above expansion we obtain
\begin{align}
	&U_{k+1}-\Big(1-\frac{r}{k}\Big)U_k\notag\\
	=&U_k\Big(\frac{U_{k+1}}{U_k}-1+\frac{r}{k}\Big)\notag\\
	=&U_k\Big(\frac{r-p}{k}+\frac{\tau}{k\ln k}+s_k\Big)\notag\\
	=&A\frac{(\ln k)^\tau}{k^{p+1}}\Big(r-p+\frac{\tau}{\ln k}+k s_k\Big). \label{u1_1}
\end{align}
Denote $T_k=r-p+\frac{\tau}{\ln k}+k s_k$. Because $|k s_k|\le C_4/k$ and $\frac{\tau}{\ln k}\to0$, we have $\lim_{k\to\infty}T_k=r-p>0$. Consequently, there exists $K_g\ge\max\{k_0,K_f,2\}$ such that for all $k\ge K_g$,
\begin{align}
	T_k\ge\frac{r-p}{2}>0. \label{u1_2}
\end{align}

Let
\begin{align}
	A=\max\Big\{\frac{2d}{r-p},\ \frac{Y_{K_g}K_g^p}{(\ln K_g)^\tau}\Big\}+1.
\end{align}
Thus, the following conclusion holds:
\begin{itemize}
	\item[(i)] For all $k\ge K_g$, using \eqref{u1_1}, \eqref{u1_2} and  $A\ge\frac{2d}{r-p}$ gives
	\begin{align}
		U_{k+1}-\Big(1-\frac{r}{k}\Big)U_k\ge A\frac{(\ln k)^\tau}{k^{p+1}}\cdot\frac{r-p}{2}\ge \frac{d(\ln k)^\tau}{k^{p+1}}.
	\end{align}
	\item[(ii)] At the initial index $k=K_g$,
	\begin{align}
		U_{K_g}=A\frac{(\ln K_g)^\tau}{K_g^p}\ge Y_{K_g}.
	\end{align}
\end{itemize}

Define $W_k:=Y_k-U_k$ for $k\ge K_g$. Thus,  $W_{K_g} =Y_{K_g}-U_{K_g}\le0$.

Using \eqref{re_Y} for $Y_k$ and property~(i), we obtain for all $k\ge K_g$,
\begin{align}
	W_{k+1}&=Y_{k+1}-U_{k+1}\notag\\
	&\le\Big(1-\frac{r}{k}\Big)Y_k+\frac{d(\ln k)^\tau}{k^{p+1}}-U_{k+1}\notag\\
	&=\Big(1-\frac{r}{k}\Big)(W_k+U_k)+\frac{d(\ln k)^\tau}{k^{p+1}}-U_{k+1}\notag\\
	&=\Big(1-\frac{r}{k}\Big)W_k-\Big[U_{k+1}-\Big(1-\frac{r}{k}\Big)U_k-\frac{d(\ln k)^\tau}{k^{p+1}}\Big]\notag\\
	&\le\Big(1-\frac{r}{k}\Big)W_k.
\end{align}
Since $1-\frac{r}{k}\ge0$ for all large $k$ (which can be ensured by taking $K_g\geq r$) and $W_{K_g}\le0$, we obtain that $W_k\le0$ for every $k\ge K_g$. Hence
\begin{align}
	Y_k\le U_k=A\frac{(\ln k)^\tau}{k^p},\qquad\forall k\ge K_g.
\end{align}

\subsection{Proof of Theorem \ref{thm:convergence_rate}} \label{app_theorem2}
	From \eqref{eq1c}, we know that for almost every sample 
path \(\omega\), there exists a finite integer \(K_h(\omega)\) such that for all \(k \ge K_h(\omega)\),
\begin{align}
	V_{k+1} \le \bigl(1 - \mu \alpha_k \bigr) V_k + \beta_k,
\end{align}
where \(V_k = \Vert \mathbf{x}_k - x^* \Vert^2\)  and
\begin{align}
	\beta_k = C_3 \alpha_k \bar{\eta}_k^2 + 2 G^2 \alpha_k^2,
\end{align}
with deterministic constants \(C_3, G\) independent of \(k\) and \(\omega\). By \eqref{eta}, 
\begin{align}
	\bar{\eta}_k^2 = N C_2^2 \left( \frac{\ln(k+1)}{k^\beta} \right)^{\frac{2(\delta-1)}{\delta}}, \label{eta2}
\end{align}
where \(C_2 = C_1 \cdot 48^{(\delta-1)/\delta}\) and \(C_1 = (12\nu^\delta)^{1/\delta}\).

Since \(\alpha_k = b/(k+1)\),  we have \(1 - \mu\alpha_k= 1 - \frac{\mu b}{k} + \frac{\mu b}{k(k+1)}\leq 1 - \frac{\mu b}{k} + \frac{\mu b}{k^2} \).

For all \(k \ge K_h(\omega)\),
\begin{align}
	V_{k+1} &\le \Bigl( 1 - \frac{\mu b}{k} + \frac{\mu b}{k^2} \Bigr) V_k + \beta_k \nonumber \\
	&= \Bigl( 1 - \frac{\mu b}{k} \Bigr) V_k + \frac{\mu b}{k^2} V_k + \beta_k. \label{eq:rec2}
\end{align}

By Assumption~1, the constraint set \(\Omega\) is compact. Since \(\mathbf{x}_k \in \Omega\) and 
\(x^* \in \Omega\), there exists a deterministic constant \(D > 0\) such that for all \(k\),
\begin{align}
	V_k = \Vert \mathbf{x}_k - x^* \Vert^2 \le D. \label{eq:Vbound}
\end{align}
This bound holds almost surely for every sample path.

Define a  deterministic constant
\begin{align}
	C_5 := \mu b D + 2 G^2 b^2. \label{eq:C7}
\end{align}

Thus,
\begin{align}
	\frac{\mu b}{k^2} V_k + 2 G^2 \alpha_k^2 \le \frac{C_5}{k^2}. \label{eq:combine}
\end{align}
Substituting \eqref{eq:combine} and the remaining part of \(\beta_k\) into \eqref{eq:rec2} gives
\begin{align}
	V_{k+1} \le \Bigl( 1 - \frac{\mu b}{k} \Bigr) V_k + C_3 \alpha_k \bar{\eta}_k^2 + \frac{C_5}{k^2},  k \ge K_h(\omega). \label{eq:rec3}
\end{align}

Using \(\alpha_k \le b/k\) and  \eqref{eta2}, we obtain for \(k \ge K_h(\omega)\),
\begin{align}
	C_3 \alpha_k \bar{\eta}_k^2 &\le C_3 \frac{b}{k} \cdot N C_2^2 \left( \frac{\ln(k+1)}{k^\beta} \right)^{\frac{2(\delta-1)}{\delta}} \nonumber \\
	&= b C_3 N C_2^2 \cdot \frac{\bigl( \ln(k+1) \bigr)^{\frac{2(\delta-1)}{\delta}}}{k^{1 + \beta \cdot \frac{2(\delta-1)}{\delta}}}. \label{eq:expand}
\end{align}
For \(k \ge 3\), we have \(\ln(k+1) \le \ln(k^2) = 2 \ln k\). Thus,
\begin{align}
	\bigl( \ln(k+1) \bigr)^{\frac{2(\delta-1)}{\delta}} \le 2^{\frac{2(\delta-1)}{\delta}} (\ln k)^{\frac{2(\delta-1)}{\delta}}.
\end{align}

Let
\begin{align}
	\theta:= \beta \cdot \frac{2(\delta-1)}{\delta}, \qquad 
	\varpi: = \frac{2(\delta-1)}{\delta}. \label{eq:gammanu}
\end{align}
and 
\begin{align}
	C_6 := b C_3 N C_2^2 \cdot 2^{\varpi}. \label{eq:C6}
\end{align}
Hence, for all \(k \ge K_l(\omega) :=\max\{K_h(\omega), 3\}\),
\begin{align}
	C_3 \alpha_k \bar{\eta}_k^2 \le C_6 \frac{(\ln k)^\varpi}{k^{\theta+1}}. \label{eq:w1a}
\end{align}

Combining \eqref{eq:rec3} and \eqref{eq:w1a}, we obtain 
for all \(k \ge K_l(\omega)\),
\begin{align}
	V_{k+1} \le \Bigl( 1 - \frac{\mu b}{k} \Bigr) V_k + w_k^{(1)} + w_k^{(2)}, \label{eq:finalrec}
\end{align}
where
\begin{align}
	w_k^{(1)} := C_6 \frac{(\ln k)^\varpi}{k^{\theta+1}}, \qquad
	w_k^{(2)} := \frac{C_5}{k^2}. \label{eq:weights}
\end{align}

Let $ K_m(\omega):=\max\{K_l(\omega),\mu b \}$. We define two auxiliary sequences \(\{Z_k^{(1)}\}_{k \ge K_m(\omega)}\) and 
\(\{Z_k^{(2)}\}_{k \ge K_m(\omega)}\) recursively as follows:
\begin{align}
	Z_{K_m(\omega)}^{(1)} &= V_{K_m(\omega)}, \quad 
	Z_{k+1}^{(1)} = \Bigl( 1 - \frac{\mu b}{k} \Bigr) Z_k^{(1)} + w_k^{(1)},   \label{eq:aux1} \\
	Z_{K_m(\omega)}^{(2)} &= 0, \quad 
	Z_{k+1}^{(2)} = \Bigl( 1 - \frac{\mu b}{k} \Bigr) Z_k^{(2)} + w_k^{(2)}. \label{eq:aux2}
\end{align}

By induction, we claim that for all \(k \geq K_m(\omega)\),
\begin{align}
	V_k \le Z_k^{(1)} + Z_k^{(2)}. \label{eq:comparison}
\end{align}
For \(k = K_m(\omega)\), we have
\begin{align}
	Z_{K_m(\omega)}^{(1)} + Z_{K_m(\omega)}^{(2)} = V_{K_m(\omega)},
\end{align}
so \eqref{eq:comparison} holds with equality.
Assume that \(V_k \le Z_k^{(1)} + Z_k^{(2)}\) for some \(k \ge K_m(\omega)\).
From \eqref{eq:finalrec} and the induction hypothesis, we obtain
\begin{align}
	V_{k+1} 
	&\le \Bigl( 1 - \frac{\mu b}{k} \Bigr) \bigl( Z_k^{(1)} + Z_k^{(2)} \bigr) + w_k^{(1)} + w_k^{(2)} \nonumber \\
	&= \Bigl[ \Bigl( 1 - \frac{\mu b}{k} \Bigr) Z_k^{(1)} + w_k^{(1)} \Bigr] 
	+ \Bigl[ \Bigl( 1 - \frac{\mu b}{k} \Bigr) Z_k^{(2)} + w_k^{(2)} \Bigr] \nonumber \\
	&= Z_{k+1}^{(1)} + Z_{k+1}^{(2)}, k\geq K_m(\omega)
\end{align}
where the first inequality uses the fact that \(1 - \frac{\mu b}{k} \ge 0\) for $k\geq K_m(\omega)$. This completes the induction, establishing \eqref{eq:comparison} 
for all \(k \ge K_m(\omega)\).

The recursion \eqref{eq:aux1} matches the form of Lemma \ref{lem:chung-log} with parameters 
\(r = \mu b\), \(p = \theta\), \(d = C_6\), \(\tau = \varpi\). 
By the hypothesis of Theorem \ref{thm:convergence_rate}, we have selected \(b\) such that \(\mu b > \theta\). 
Therefore, Lemma \ref{lem:chung-log} guarantees the existence of a  constant \(A_1(\omega) > 0\)  and   $K_p(\omega)\geq K_m(\omega)$ such that $\forall k \ge K_p(\omega)$,
\begin{align}
	Z_k^{(1)} \le A_1(\omega) \frac{(\ln k)^{\varpi}}{k^{\theta}} 
	=A_1(\omega) \left( \frac{\ln k}{k^{\beta}} \right)^{\frac{2(\delta-1)}{\delta}}. \label{eq:bound1}
\end{align}

The recursion \eqref{eq:aux2} corresponds to Lemma \ref{lem:chung-log} with \(r = \mu b\), \(p = 1\), 
\(d = C_5\), \(\tau = 0\). Since  \(\mu b > 1\), Lemma \ref{lem:chung-log} yields a constant \(A_2 (\omega)> 0\) and   $K_q(\omega)\geq K_m(\omega)$ such that $ \forall k \ge K_q(\omega)$,
\begin{align}
	Z_k^{(2)} \le A_2(\omega) \frac{1}{k}. \label{eq:bound2}
\end{align}

Combining the comparison inequality \eqref{eq:comparison} with \eqref{eq:bound1} and 
\eqref{eq:bound2}, we obtain  \(\forall k \ge K_r(\omega):=\max\{K_p(\omega),K_q(\omega)\} \),
\begin{align}
	V_k 
	\le A_1 (\omega)\left( \frac{\ln k}{k^{\beta}} \right)^{\frac{2(\delta-1)}{\delta}} + A_2(\omega) \frac{1}{k}.
\end{align}

Thus, there exists a 
constant \(A_3(\omega): = 2 \max\{A_1(\omega), A_2(\omega)\}\) such that
\begin{align}
	V_k \le A_3(\omega) \max\left\{ \frac{1}{k}, \left( \frac{\ln k}{k^{\beta}} \right)^{\frac{2(\delta-1)}{\delta}} \right\}, 
	\qquad \forall k \ge K_r(\omega). \label{eq:combined}
\end{align}

This completes the proof of the main convergence rate statement.

When \(\delta = 2\), the noise has finite variance. Then
\begin{align}
	\frac{2(\delta-1)}{\delta} =1, \qquad \varpi = 1, \qquad \theta = \beta.
\end{align}
If  \(\beta \ge 1\), then for all \(k \ge 3\),
\begin{align}
	\max\left\{ \frac{1}{k}, \frac{\ln k}{k^{\beta}} \right\} \le \frac{\ln k}{k}.
\end{align}

Thus,
\begin{align}
	V_k \le A_3 (\omega) \frac{\ln k }{k}, \forall k \ge K_r(\omega). \label{eq:combined2}
\end{align}

\subsection{Proof of Theorem \ref{thm:online}} \label{appd_V1}
	Let $V_k \coloneqq \Vert\mathbf{x}_k - x^*\Vert^2$.
For each player $i$, let
\begin{align}
	\epsilon_{i,k}^{\mathrm{MoM}} & = \hat{g}_{i,k} - \nabla_{x_i}J_i(\mathbf{x}_k), \label{eq:emom}\\
	\bar{\epsilon}_{i,k} &\coloneqq \bar{g}_{i,k} - \nabla_{x_i}J_i(\mathbf{x}_k). \label{eq:ebar}
\end{align}

Hence, according to \eqref{step4_1},
\begin{align}
	\epsilon_{i,k}^{\mathrm{Bias}} \coloneqq \tilde{g}_{i,k} - \nabla_{x_i}J_i(\mathbf{x}_k)
	= (1-\eta_k)\epsilon_{i,k}^{\mathrm{MoM}} + \eta_k\bar{\epsilon}_{i,k}. \label{eq:ecorr}
\end{align}

Lemma \ref{lem:cond_tail_bound} still holds for this algorithm. For every player $i$ and iteration $k$,
\begin{align}
	\mathbb{P}\bigl(|\epsilon_{i,k}^{\mathrm{MoM}}| > \mathcal{E}_k(\gamma_k) \mid \mathcal{F}_k\bigr) \le 2 \gamma_k \quad \text{a.s.}, \label{eq:mom_prob}
\end{align}
where
\begin{align}
\mathcal{E}_k(\gamma_k) = C_1\Bigl(\frac{16\ln(e^{1/8}\gamma_k^{-1})}{m_k}\Bigr)^{\frac{\delta-1}{\delta}},\qquad
C_1 = (12\nu^\delta)^{1/\delta}.
\end{align}
Since $2\sum_k \gamma_k < \infty$, following a similar analysis as \eqref{kesi1}, the Borel--Cantelli lemma implies that for almost every sample path $\omega$ there exists a finite integer $\kappa _a(\omega)$ such that for all $k\ge \kappa_a(\omega)$ and all $i$,
\begin{align}
	|\epsilon_{i,k}^{\mathrm{MoM}}| \le \mathcal{E}_k = C_2\Bigl(\frac{\ln(k+1)}{k^\beta}\Bigr)^{\frac{\delta-1}{\delta}}, \label{eq:mom_bound}
\end{align}
where $C_2= C_1 \cdot 48^{(\delta-1)/\delta}$.

Let $m_k':=b_ks_k $ is the number of the real adopted samples for player $i$ at iteration $k$.  Since	$s_k=\lfloor \frac{m_k}{b_k}\rfloor$, $s_k\geq \frac{m_k}{b_k}-1$, which implies that   $m_k'\ge m_k-b_k\geq  \frac{1}{2} m_k$.  Moreover, since $m_k\ge c(k+1)^\beta$, $m_k' \ge \frac{c}{2} k^\beta$. 

Applying the von Bahr--Esseen inequality  yields
\begin{align}
	\mathbb{E}\bigl[|\bar{\epsilon}_{i,k}|^\delta \mid \mathcal{F}_k\bigr]
	\le 2\nu^\delta (m_k')^{1-\delta}
	\le C_\epsilon k^{-\beta(\delta-1)} \quad \text{a.s.},  \label{eq:mean_moment}
\end{align}
with $C_\epsilon = 2\nu^\delta (\frac{c}{2})^{1-\delta}$.

Choose any $\zeta$ satisfying $0<\zeta < \frac{\beta(\delta-1)-1}{\delta}$, which is ensured   by $\beta(\delta-1)>1$.
Define the events $A'_{i,k} := \bigl\{|\bar{\epsilon}_{i,k}| > k^{-\zeta}\bigr\}$. Using the conditional Markov inequality and \eqref{eq:mean_moment},
\begin{align}
	\mathbb{P}(A_{i,k}\mid\mathcal{F}_k)
	\le \frac{\mathbb{E}[|\bar{\epsilon}_{i,k}|^\delta\mid\mathcal{F}_k]}{k^{-\zeta\delta}}
	\le C_\epsilon k^{-\beta(\delta-1)+\zeta\delta} \quad \text{a.s.}
\end{align}
The exponent satisfies $-\beta(\delta-1)+\zeta\delta < -1$ by the choice of $\zeta$. Thus, $\sum_{k=1}^\infty k^{-\beta(\delta-1)+\zeta\delta}<\infty$. Taking  the unconditional expectation  gives $ \sum_{k=0}^\infty  \mathbb{P}(A_{i,k})<\infty$. The Borel--Cantelli lemma implies that for almost every $\omega$ there exists $\kappa_b(\omega)\ge \kappa_a(\omega)$ such that for all $k\ge \kappa_b(\omega)$ and all $i$,
\begin{align}
	|\bar{\epsilon}_{i,k}| \le k^{-\zeta}. \label{eq:mean_bound}
\end{align}

From \eqref{eq:ecorr}, \eqref{eq:mom_bound}, \eqref{eq:mean_bound}, and $0\le\eta_k\le1$, we obtain for $k\ge \kappa_b(\omega)$,
\begin{align}
	|\epsilon_{i,k}^{\mathrm{Bias}}| \le \mathcal{E}_k + \eta_k k^{-\zeta}.
\end{align}

Denote the vector $\boldsymbol{\epsilon}_k^{\mathrm{Bias}} = [\epsilon_{1,k}^{\mathrm{Bias}},\dots,\epsilon_{N,k}^{\mathrm{Bias}}]^{\top}$.
Thus,  for all $k\ge \kappa_b(\omega)$,
\begin{align}
	\Vert\boldsymbol{\epsilon}_k^{\mathrm{Bias}}\Vert^2 \le 2N\bigl(\mathcal{E}_k^2 + \eta_k^2 k^{-2\zeta}\bigr). \label{eq:err_sq}
\end{align}

Exactly as in the derivation of 	\eqref{eq:V_rec} in the proof of Theorem \ref{thm:convergence}, there exist constants $C_3>0,\,G>0$ and an index $K_d$ such that for all $k\ge K_d$,
\begin{align}
	V_{k+1} \le (1-\mu\alpha_k)V_k + C_3\alpha_k\Vert\boldsymbol{\epsilon}_k^{\mathrm{Bias}}\Vert^2 + 2G^2\alpha_k^2. \label{eq:V_recur}
\end{align}
The difference with Theorem \ref{thm:convergence} is only in the notation $\boldsymbol{\epsilon}_k^{\mathrm{Bias}}$.

Define
\begin{align}
	\beta_k' \coloneqq 2 C_3\alpha_k\bigl(\bar{\eta}_k^2 + N\eta_k^2 k^{-2\zeta}\bigr) + 2G^2\alpha_k^2, \label{eq:betak}
\end{align}
where $
	\bar{\eta}_k:=\sqrt{N} \mathcal{E}_k .
$
Using \eqref{eq:err_sq} in \eqref{eq:V_recur}, we obtain that, for all $k\ge \kappa_c(\omega):=\max\{\kappa_b(\omega),K_d\}$,
\begin{align}
	V_{k+1} \le (1-\mu\alpha_k)V_k + \beta'_k. \label{eq:recur_final}
\end{align}

We now verify the three conditions of Lemma \ref{a}.
First, $\sum_{k=0}^\infty \alpha_k = \infty$. Second,  there exists $K_e$ such that $0<\mu\alpha_k\leq1$ for $k\geq K_e$. Moreover,  
\begin{align}
	\frac{\beta'_k}{\alpha_k}
	=2NC_3 \bigl(\mathcal{E}_k^2 + \eta_k^2 k^{-2\zeta}\bigr) + 2G^2\alpha_k, \label{eq:ratio}
\end{align}
which implies that
$\displaystyle\lim_{k\to\infty}\frac{\beta'_k}{\alpha_k}=0$.

Thus, by Lemma~1, $V_k\to0$ on almost every sample path, i.e.
\begin{align}
\lim_{k\to\infty}\Vert\mathbf{x}_k - x^*\Vert = 0 \quad \text{a.s.}
\end{align}

\subsection{Proof of Theorem \ref{thm:rate_online}} \label{appd_th4}

According to \eqref{eq:recur_final}, for almost every sample path $\omega$ there exists an index $\kappa_c(\omega)$ such that for all $k\ge \kappa_c(\omega)$,
\begin{align}
	V_{k+1} \le (1-\mu\alpha_k) V_k + \beta_k', \label{eq:rate_rec}
\end{align}
where
\begin{align}
	\beta_k' = 2 C_3 \alpha_k\bigl( \bar{\eta}_k^2 + N\eta_k^2 k^{-2\zeta} \bigr) + 2G^2 \alpha_k^2. \label{eq:rate_beta1}
\end{align}

According to a similar analysis as  \eqref{eq:w1a}, there exists  $\kappa_l(\omega):=\max\{\kappa_c(\omega),3\}$ such that $\forall k \geq \kappa_l(\omega)$,
\begin{align}
	C_3 \alpha_k \bar{\eta}_k^2 \le C_6 \frac{(\ln k)^\varpi}{k^{\theta+1}}, \label{eq:w1_0}
\end{align}
where 
\begin{align}
	C_6  = b C_3 N C_2^2 \cdot 2^{\varpi}. \label{eq:C6_2}
\end{align}
and 
\begin{align}
	\theta = \beta \cdot \frac{2(\delta-1)}{\delta}, \qquad 
	\varpi  = \frac{2(\delta-1)}{\delta}. \label{eq:gammanu_2}
\end{align}

For the second term,
\begin{align}
	2N C_3 \alpha_k \eta_k^2 k^{-2\zeta}
	\le& 2N C_3 \frac{b}{k} \cdot \eta_0^2 (k+1)^{-2\rho} k^{-2\zeta} \notag\\
	\le &C_{\eta} \frac{1}{k^{1+2\rho+2\zeta}}, \label{eq:bound_eta_1}
\end{align}
with $C_{\eta}: = 2N C_3 b \eta_0^2$.

For the third term, 
\begin{align}
	2G^2 \alpha_k^2 \le 2G^2 \frac{b^2}{k^2}. \label{eq:bound_step_1}
\end{align}

Similar to \eqref{eq:rec2}--\eqref{eq:rec3}, we can obtain that
\begin{align}
	V_{k+1} \le &\Bigl( 1 - \frac{\mu b}{k} \Bigr) V_k + 2 C_3 \alpha_k\bigl( \bar{\eta}_k^2 \notag\\
	& + N\eta_k^2 k^{-2\zeta} \bigr) + \frac{C_5}{k^2},  k \ge \kappa_c(\omega), \label{eq:rec3_2}
\end{align}
where $ 	C_5 = \mu b D + 2 G^2 b^2. $

Substituting \eqref{eq:w1_0}, \eqref{eq:bound_eta_1}, \eqref{eq:bound_step_1} into \eqref{eq:rec3_2} gives, for all $k\ge \kappa_l(\omega)$,
\begin{align}
	V_{k+1} \le \Bigl(1-\frac{\mu b}{k}\Bigr) V_k + w_k^{(1)} + w_k^{(2)} + w_k^{(3)}, \label{eq:rec_split}
\end{align}
where
\begin{align}
	w_k^{(1)} &= 2C_{6} \frac{(\ln k)^{\varpi}}{k^{\theta+1}}, \label{eq:w1} \\
	w_k^{(2)} &= C_{\eta} \frac{1}{k^{1+2\rho+2\zeta}}, \label{eq:w2} \\
	w_k^{(3)} &= \frac{C_5}{k^2}. \label{eq:w3}
\end{align}

Let $\kappa_m(\omega) \coloneqq \max\{\kappa_l(\omega), \mu b\}$, so that $1-\frac{\mu b}{k} \ge 0$ for all $k\ge \kappa_m(\omega)$.
Define three auxiliary nonnegative sequences $\{Z_k^{(1)}\}_{k\ge \kappa_m(\omega)}$, $\{Z_k^{(2)}\}_{k\ge \kappa_m(\omega)}$, and $\{Z_k^{(3)}\}_{k\ge \kappa_m(\omega)}$ by
\begin{align}
	Z_{\kappa_m(\omega)}^{(1)} &= V_{\kappa_m(\omega)}, & Z_{k+1}^{(1)} &= \Bigl(1-\frac{\mu b}{k}\Bigr) Z_k^{(1)} + w_k^{(1)}, \label{eq:Z1_def} \\
	Z_{\kappa_m(\omega)}^{(2)} &= 0, & Z_{k+1}^{(2)} &= \Bigl(1-\frac{\mu b}{k}\Bigr) Z_k^{(2)} + w_k^{(2)}, \label{eq:Z2_def} \\
	Z_{\kappa_m(\omega)}^{(3)} &= 0, & Z_{k+1}^{(3)} &= \Bigl(1-\frac{\mu b}{k}\Bigr) Z_k^{(3)} + w_k^{(3)}. \label{eq:Z3_def}
\end{align}

Similarly, by induction on $k$, one can verify that $V_k \le Z_k^{(1)} + Z_k^{(2)} + Z_k^{(3)}$ for all $k\ge \kappa_m(\omega)$.

For $Z_k^{(1)}$, the recursion \eqref{eq:Z1_def} matches Lemma~\ref{lem:chung-log} with parameters $r = \mu b$, $p = \theta$, $d = 2C_6$, $\tau = \varpi$.
By hypothesis, $\mu b > \theta$. Hence,  there exists an index $\kappa_p(\omega) \ge \kappa_m(\omega)$ and a constant $B_1(\omega)>0$ such that
\begin{align}
	Z_k^{(1)} \le B_1(\omega) \frac{(\ln k)^{\varpi}}{k^{\theta}} = B_1(\omega) \Bigl( \frac{\ln k}{k^{\beta}} \Bigr)^{\frac{2(\delta-1)}{\delta}}, \quad \forall k\ge \kappa_p(\omega). \label{eq:Z1_rate}
\end{align}
For $Z_k^{(2)}$, we have $r = \mu b$, $p = 2\rho+2\zeta$, $d = C_{\eta}$, $\tau = 0$. Since $\mu b > 2\rho+2\zeta$ by the choice of $b$, Lemma~\ref{lem:chung-log} provides $\kappa_q(\omega)\ge \kappa_m(\omega)$ and $B_2(\omega)>0$ with
\begin{align}
	Z_k^{(2)} \le B_2(\omega) \frac{1}{k^{2\rho+2\zeta}}, \quad \forall k\ge \kappa_q(\omega). \label{eq:Z2_rate}
\end{align}
For $Z^{(3)}$, the parameters are $r = \mu b$, $p = 1$, $d = C_5$, $\tau = 0$. Because $\mu b > 1$ and by  Lemma~\ref{lem:chung-log}, there exists  $\kappa_s(\omega)\ge \kappa_m(\omega)$ and $B_3(\omega)>0$ such that
\begin{align}
	Z_k^{(3)} \le B_3(\omega) \frac{1}{k}, \quad \forall k\ge \kappa_s(\omega). \label{eq:Z3_rate}
\end{align}

Let $\kappa_r(\omega) \coloneqq \max\{\kappa_p(\omega), \kappa_q(\omega), \kappa_s(\omega)\}$ and $B(\omega) \coloneqq 3\max\{B_1(\omega), B_2(\omega), B_3(\omega)\}$.
Then for all $k\ge \kappa_r(\omega)$,
\begin{align}
	V_k &\le B_1(\omega) \Bigl( \frac{\ln k}{k^{\beta}} \Bigr)^{\frac{2(\delta-1)}{\delta}}
	+ B_2(\omega) \frac{1}{k^{2\rho+2\zeta}}
	+ B_3(\omega) \frac{1}{k} \nonumber \\
	&\le B(\omega) \max\Bigl\{ \frac{1}{k}, \Bigl(\frac{\ln k}{k^{\beta}}\Bigr)^{\frac{2(\delta-1)}{\delta}}, \frac{1}{k^{2\rho+2\zeta}} \Bigr\}. \label{eq:final_rate}
\end{align}

When $\delta=2$, $\varpi = 1$, $\theta = \beta$.
If $\beta\ge 1$ and $2\rho+2\zeta \ge 1$, then for   $k\geq 3$,
\begin{align}
	\max\Bigl\{ \frac{1}{k}, \frac{\ln k}{k^{\beta}}, \frac{1}{k^{2\rho+2\zeta}} \Bigr\} \le \frac{\ln k}{k},
\end{align}
which yields \eqref{eq:rate_online_special}. This completes the proof.

\end{document}